\documentclass[a4paper,12pt]{amsart}
\usepackage[utf8]{inputenc}

\usepackage{amsmath,amssymb,amsthm}
\usepackage{amsaddr}
\usepackage{bbm}
\usepackage{bm}
\usepackage{xspace}
\usepackage{csquotes}
\usepackage{hyperref} % ,cleveref}
\usepackage{makecell}

\usepackage[x11names,table]{xcolor}
\usepackage{graphicx}
\graphicspath{ {img/} }
\usepackage{tikz}
\usetikzlibrary{calc}
\usetikzlibrary{cd}
\usetikzlibrary{3d}
\usepackage{csquotes}
\usepackage{booktabs}
\hypersetup{
	colorlinks = true,          %Colors links instead of ugly boxes
	urlcolor = DeepSkyBlue3, linkcolor = Chartreuse4, citecolor = DarkOrange2
}
\usepackage{fullpage}
\RequirePackage[backend=bibtex,isbn=false,url=false,maxbibnames=5,bibencoding=utf8]{biblatex}
\renewbibmacro{in:}{, }
\DeclareFieldFormat[misc]{title}{\mkbibquote{#1}}
\DeclareFieldFormat[article]{volume}{\mkbibbold{#1}}
\DeclareFieldFormat{url}{\textsc{url}: \href{#1}{#1}}
\DeclareFieldFormat{doi}{\textsc{doi}: \href{http://dx.doi.org/#1}{#1}}
\DeclareFieldFormat{eprint:arxiv}{arXiv:\href{https://arxiv.org/abs/#1}{#1}}
\newcommand\OSCAR{\texttt{OSCAR}\xspace}

\newcommand\HC{\texttt{Homotopy\-Continuation.jl}\xspace}
\newcommand\msolve{\texttt{msolve}\xspace}
\newcommand\NN{\mathbb{N}}
\newcommand\ZZ{\mathbb{Z}}
\newcommand\RR{\mathbb{R}}
\newcommand\CC{\mathbb{C}}
\newcommand\CP{\mathbb{P}} % complex projective space
\newcommand\RP{\mathbb{RP}} % real projective space
\newcommand\Sph{\mathbb{S}} % sphere
\newcommand\cA{\mathcal{A}}

\DeclareMathOperator{\conv}{conv}
\DeclareMathOperator{\nvol}{nvol}
\DeclareMathOperator{\SL}{SL} % special linear group

\renewcommand\phi{\varphi} % if really necessary

\theoremstyle{theorem}
\newtheorem{theorem}{Theorem}
\newtheorem{theorem*}[theorem]{Theorem*}
\newtheorem{lemma}[theorem]{Lemma}
\newtheorem{proposition}[theorem]{Proposition}
\newtheorem{corollary}[theorem]{Corollary}

\theoremstyle{definition}
\newtheorem{example}[theorem]{Example}
\theoremstyle{remark}
\newtheorem{remark}[theorem]{Remark}

\newcommand{\setof}[2]{\{#1\;|\;#2\}}

\title{Wronski Pairs of Honeycomb Curves}

\author{Laura Casabella}
\address[Laura Casabella]{
  \vspace{-0.3cm}Max-Planck Institute for Mathematics in the Sciences, Leipzig, Germany\\
  \texttt{laura.casabella@mis.mpg.de}
}

\author{Michael Joswig}
\address[Michael Joswig]{
  \vspace{-0.3cm}Chair of Discrete Geometry/Mathematics\\
  Technische Universität Berlin, Germany and\\
  Max-Planck Institute for Mathematics in the Sciences, Leipzig, Germany \\
  \texttt{joswig@math.tu-berlin.de}
}

\author{Rafael Mohr} 
\address[Rafael Mohr]{
  \vspace{-0.3cm}Inria Saclay, Palaiseau, France\\
  \texttt{rafael.mohr@inria.fr}
}
\thanks{The third author is the corresponding author.}

\begin{document}

\begin{abstract}
  We study certain generic systems of real polynomial equations associated with triangulations of convex polytopes and investigate their number of real solutions.
  Our main focus is set on pairs of plane algebraic curves which form a so-called Wronski system.
  The computational tasks arising in the analysis of such Wronski pairs lead us to the frontiers of current computer algebra algorithms and their implementations, both via Gröbner bases and numerical algebraic geometry.
\end{abstract}
\smallskip

\begingroup
\def\uppercasenonmath#1{} % this disables uppercasing title
\let\MakeUppercase\relax % this disables uppercasing authors
\maketitle
\endgroup

Numerical Algebraic Geometry; Gröbner Bases
\section{Introduction}

Our objects of study are systems of polynomial equations with finitely many solutions.
When those systems come in parametric families it is desirable to obtain bounds for the number of solutions which are uniform for the entire family.
There is a very satisfying theory over the complex numbers, which has been established by Bernstein, Kushnirenko and Khovanskii; see \cites[Chapter~3]{solving}[\S7.5]{CoxLittleOShea:2005} for textbook references.
They observed that natural bounds of this kind are encoded in metric properties of the Newton polytopes of the polynomials involved.
It should be stressed that this theory goes beyond mere existence results.
In fact, the method of homotopy continuation offers competitive algorithms for solving systems of complex polynomial equations based on this theory.
Implementations include \texttt{PHCpack} \cite{phcpack}, \texttt{Bertini} \cite{bertini} and \HC \cite{breiding2018}.
Yet, in many practical applications only the real solutions matter.
This observation led to a quest for real versions of the theorems of Bernstein, Kushnirenko and Khovanskii.
As a breakthrough Soprunova and Sottile \cite{SoprunovaSottile:2006} obtained a real version of Kushnirenko's theorem for so-called \emph{Wronski systems}; see also \cite{Sottile:2011}.
The latter result is concerned with the unmixed case, where the Newton polytopes of all polynomials of the system agree.
Obtaining nontrivial lower bounds for the number of real solutions of a polynomial system is not so easy, for the simple reason that many systems do not have any real roots at all.
The main idea of \cite{SoprunovaSottile:2006} is to study the real toric variety associated to an unmixed system of $d$ real polynomials in $d$ variables.
When the Newton polytope admits a very special kind of triangulation (it needs to be regular, dense and foldable), then that real toric variety admits a map to the real projective space $\RP^{d-1}$.
By passing to double covers we obtain a map between smooth submanifolds which are often orientable (and satisfy an extra condition); then the degree of that map serves as a lower bound for the number of real roots.

It has been stressed in a list of open problems \cite[\S7.4]{Sottile:2011} that some aspects of the theory remain
unclear, probably because only a small number of examples has been worked out in detail.  So, one purpose of this paper
is to study the applicability of the method developed by Soprunova and Sottile \cite{SoprunovaSottile:2006} to pairs
of plane algebraic curves with full support of a given degree. In that case the common Newton polygon of the two curves
is a scaled unit triangle $\delta\Delta_2=\conv\{(0,0),(\delta,0),(0,\delta)\}$, where $\delta$ is the degree.  The triangle
$\delta\Delta_2$ admits a particular regular, dense and foldable triangulation which has been called the \emph{honeycomb
  triangulation} in \cite{BJMS:2015}.  Up to a unimodular transformation of the plane, this agrees with the
\emph{alcoved triangulation} of Lam and Postnikov \cite{Lam+Postnikov:2007}.  By fixing the Newton polygon
$\delta\Delta_2$ and the honeycomb triangulation, in order to apply the results from \cite{SoprunovaSottile:2006}, the only
nontrivial ingredient left are the \enquote{extra conditions} mentioned above. These conditions amount to checking that
specific systems of three polynomial equations in three indeterminates do not have undesirable solutions.  These
\emph{meta-systems}, as we call them, depend on the degree $\delta$ and the choice of a height function inducing the
honeycomb triangulation. The meta-systems are highly structured and turn out to be computationally challenging for
higher degrees. In this paper we attempt to exhibit the frontiers of computation for current versions of
\HC and the symbolic library \msolve \cite{berthomieu2021} on these meta-systems.

The paper is organized as follows.
Section~\ref{sec:wronski} summarizes results of \cite{SoprunovaSottile:2006} and \cite[Chapter 7]{Sottile:2011} to fix our notation.
One aspect that did not receive much attention in the literature so far is the fact that the specific height function for the regular triangulation plays a crucial role.
Recall that, by definition, a triangulation of a point configuration is regular if it is induced by a height function; cf.\ \cite[Chapter 5]{Triangulations}.
As our first contribution we show that a height function which satisfies the \enquote{extra conditions}, at least in the torus, always exists (Proposition~\ref{prop:resultant}).
The argument makes use of sparse resultants; cf.\ \cite{gelfand1994,bender2024}.
Assuming that the triangulation satisfies an additional local smoothness condition, we can show that the \enquote{extra conditions} are satisfied globally (Theorem~\ref{thm:lifting-exists}).
The Section~\ref{sec:alcoved} covers the polyhedral geometry part of the story, with a focus on alcoved polytopes, honeycomb polygons and their triangulations.
Here we give an explicit description of the secondary cone of the honeycomb triangulations of $\delta\Delta_2$ for arbitrary $\delta$ (Proposition~\ref{prop:liftfunc}).
This open polyhedral cone forms the set of all admissible height functions.
The final Section~\ref{sec:compute} then collects our second contribution, which is computational.
Here we study the meta-systems, which check the \enquote{extra conditions}, directly.
In our scenario $P=\delta\Delta_2$ is a honeycomb triangle with $3\leq\delta\leq 17$.
For $\omega$ we consider two different functions, which behave quite differently.
First we pick a function which occurs, e.g., in work of Sinn and Sjöberg \cite{sinn2021}.
Our second height function is the minimal height function described in Remark~\ref{rem:minimal}.
With \HC we obtain computational results related to these \enquote{extra conditions} for odd $\delta\leq 17$ where as with
\msolve we obtain results for odd $\delta \leq 11$.

\subsection*{Acknowledgements}
We thank Matías Bender for helpful discussions concerning the proof of Proposition~\ref{prop:resultant}, and Lorenzo Baldi for valuable comments.  MJ is very
grateful for the generous hospitality of the CMAP lab at \'Ecole Polytechnique while co-writing this article.
The work of MJ is partially supported by the
Deutsche Forschungsgemeinschaft (DFG, German Research Foundation); \enquote{MaRDI (Mathematical Research Data
  Initiative)} (NFDI 29/1, project ID 460135501); The Berlin Mathematics Research Center MATH$^+$ (EXC-2046/1, project ID 390685689).
The work of RM is partially supported by the Forschungsinitiative Rheinland-Pfalz and by the
European Research Council (ERC) under the European Union’s Horizon Europe research and innovation programme;
\enquote{10000 DIGITS} (project ID 101040794).
Further, MJ and RM both received DFG support through \enquote{Symbolic Tools in Mathematics and their Application} (TRR 195, project ID 286237555).

\section{Systems of Wronski polynomials} \label{sec:wronski}

We start out by fixing our notation and by describing key results from the literature that are relevant for our work.
In fact, our exposition follows \cite[Chapter~7]{Sottile:2011}.
For further details, the reader is referred to Sottile's monograph.
Our goal is to construct systems of polynomial equations from very special triangulations of lattice polytopes.
These systems come with a natural lower bound for the number of their real solutions.

Let $P$ be a $d$-dimensional lattice polytope, and let $\cA=P\cap\ZZ^d$ be the set of lattice points of $P$.
We assume that the point configuration $\cA$ is primitive, i.e., it affinely generates the entire lattice $\ZZ^d$.
A triangulation $T$ of $P$ is \emph{dense} (or \enquote{full}) if the vertex set of $T$ agrees with $\cA$.
Further, the triangulation $T$ is \emph{foldable} if its graph (i.e., the $1$-skeleton of $T$) admits a proper coloring with $d+1$ colors.
That is the case if and only if the dual graph of $T$ is bipartite; this follows from \cite[Corollary 11]{Joswig:2002}. 
In loc.\ cit.\ and elsewhere foldable triangulations are also called \enquote{balanced}.
Note that the coloring of the vertices of a foldable triangulation is unique, up to relabeling.

So let $T$ be a dense and foldable triangulation of $P$, and let $\ell: \cA \to \{0, 1, \dots, d\}$ be a proper coloring of the points in $\cA$. 
Choosing positive constants $\kappa=(\kappa_a \mid a\in\cA)$, and $d{+}1$ real parameters $c_0,c_1,\dots,c_d$, we obtain a family of \emph{Wronski polynomials}
\begin{equation}\label{eq:wronski}
  W_{\kappa,c}(x) \ = \ \sum_{a \in \cA} c_{\ell(a)}\cdot \kappa_a x^a \enspace .
\end{equation}
where $x = (x_1,\dots,x_d)$. We do admit arbitrary integer vectors as exponents, whence Wronski polynomials are actually
Laurent polynomials.  By choosing $d$ parameter vectors $c=(c_0,c_1,\dots,c_d)$ we obtain $d$ polynomials
\eqref{eq:wronski}, and they form a \emph{system of Wronski polynomials}.  Note that the positive constants $\kappa$ are not
varied.  As the dual graph of $T$ is bipartite, we may distinguish between black and white maximal cells.  The
\emph{signature} of $T$, denoted $\sigma(T)$, is defined as the absolute difference of the number of black and white maximal
cells of $T$, among those maximal cells whose normalized volume is odd.  The normalized volume, $\nvol(\cdot)$, is $d!$
times the Euclidean volume.  Under technical conditions the signature becomes the lower bound for the number of real
solutions of a system of Wronski polynomials.

Before we proceed to explain those conditions, let us consider the univariate case $d=1$, where $P$ is an interval $[a,b]$ in the real line with integral endpoints.
The number of lattice points in $P$ equals $b-a+1$, and the $b-a$ consecutive unit intervals between them form a dense and foldable triangulation.
Fixing a positive vector $\kappa$ and a vector $c=(c_0,c_1)$ with $c_i\neq0$, we get a Wronski polynomial system with a single univariate Wronski polynomial, $W=W_{\kappa,c}\in\CC[x^\pm]$.
Since we  are in the toric setting, we are interested in the roots of $W$ in the $1$-dimensional algebraic torus $\CC^{\times}=\CC\setminus\{0\}$.
The number of nonzero roots does not change if $W$ is multiplied by $x^{-a}$.
In other words, we may assume that $a=0$, and the degree of $W$ equals $b$.
Finally, the signature, $\sigma$, equals zero if $b$ is even, and it is one if $b$ is odd.
And so $\sigma$ provides a lower bound for the number of nonzero real roots of $W$.
Observe that any univariate real polynomial with full support is a Wronski polynomial.

We return to the multivariate setting, where the situation is more involved.
Setting $m=|\cA|$ we consider the $(m{-}1)$-dimensional complex projective space $\CP^\cA$ with coordinates $(z_a \mid a \in\cA)$.
The projective toric variety $X_\cA\subset\CP^\cA$ parameterized by the monomials $(x^a \mid a \in\cA)$ is the Zariski closure of the map
\[
  \phi_\cA \,:\, (\CC^{\times})^d \to \CP^\cA \ , \quad x \mapsto (x^a \mid a \in\cA) \enspace.
\]
As $\cA$ is primitive, the map $\phi_A$ is a bijection, and the toric variety $X_A$ is normal.
The latter property forces that $X_A$ is smooth in codimension one.
Restricting to the real points yields the real toric variety $Y_\cA=X_\cA\cap\RP^\cA$.
The $(m{-}1)$-dimensional sphere $\Sph^\cA$ is the double cover of $\RP^\cA$.
The pre-image of $Y_\cA$ with respect to the covering map, is the spherical toric variety $Y_\cA^+$.

Now we make the additional assumption that the triangulation $T$ of $\cA$ is regular.
That is, there is a function $\omega:\cA\to\NN$ such that $T$ is the image of the orthogonal projection (omitting the last coordinate) of the lower convex hull of the lifted points $\setof{(a,\omega(a))}{a\in \cA}$ in $\RR^{d+1}$. 
Then, for $t\in \RR^{\times}$, we get the \emph{$t$-deformation}
\begin{equation}
  \label{eq:s-deformation}
  t.\phi_\cA\,:\, (\CC^{\times})^d \to \CP^\cA \ , \quad x \mapsto (t^{\omega(a)}\cdot x^a \mid a \in\cA)
\end{equation}
with respect to $\omega$.
The family $(t.\phi_\cA \mid 0 < t \leq 1 )$ is a toric degeneration of $X_A$.
Observe that we require $\omega$ to attain nonnegative integer values, but this is not an additional restriction on $T$. 

The $t$-deformation induces an action on the monomials of $\RR[x^a \mid a \in \cA]$, so the system \eqref{eq:wronski} is deformed to 
\begin{equation}
  \label{eq:wronski-deformation}
  W_{\kappa,c}(x) \ = \ \sum_{a \in \cA} t^{\omega(a)} c_{\ell(a)}\cdot \kappa_a x^a \enspace .
\end{equation}

As the final piece of notation, let $E_{\omega,\kappa}$ be the linear space of codimension $d+1$ in $\CP^\cA$ defined by the vanishing of the terms multiplying the coefficients $c_i$ of a Wronski polynomial, i.e.,
\begin{equation}
  \label{eq:linear-subspace}
  \Lambda_i(z) \ = \ \sum_{a\in\cA_i} \kappa_a z_a \quad \text{for } 0\leq i \leq d \enspace ,
\end{equation}
where $\cA_i:=\setof{a\in \cA}{\ell(a)=i}$.

The following is a slight generalization of \cite[Theorem 7.13]{Sottile:2011}, where it is required that the triangulation $T$ is additionally unimodular; the stronger statement that we reproduce here is proved in \cite[Theorem~3.5]{SoprunovaSottile:2006}.
\begin{theorem}\label{thm:Soprunova+Sottile}
  Suppose that the set $\cA=P\cap\ZZ^d$ of lattice points in a lattice polytope $P$ is primitive and the spherical toric variety $Y_\cA^+$ is orientable.
  Let $\omega: \cA\to\NN$ be a function inducing a regular, dense and foldable triangulation $T$ with signature $\sigma(T)$, and let $\kappa\in\RR_{>0}^\cA$.
  If there exists a (minimal) $t_0>0$ such that
  \begin{equation}\label{eq:extra}
    t ^{-1}.Y_\cA \cap E_{\omega,\kappa} \ = \ \emptyset
  \end{equation}
  for all  $t \in (0,t_0]$, then a general system of Wronski polynomials \eqref{eq:wronski-deformation} for $\omega$ and $\kappa$ has exactly\/ $\nvol(P)$ complex solutions, at least $\sigma(T)$ of which are real.
\end{theorem}

For a given height function $\omega : \cA\rightarrow \NN$ we consider the polynomials
\begin{equation}
  \label{eq:metasys}
  f_i(x) \ := \ \sum_{a\in \cA_i}t^{\omega(a)}\kappa_ax^a \quad \text{for } 0\leq i \leq d \enspace .
\end{equation}
The condition \eqref{eq:extra} satisfied if the system of equations $f_0= \cdots =f_d=0$ has no real roots for almost all values of $t$.
These are the \enquote{extra conditions} from the introduction.
The statement concerning the number of the complex roots is precisely Kushnirenko's Theorem \cite[Theorem 3.2]{Sottile:2011}.

Next we want to convert the condition on the orientability of $Y_\cA^+$ into polyhedral terms.
To this end we consider the facet description
\begin{equation}
  P \ = \ \bigl\{ x\in\RR^d \,\bigm|\, U x \geq -b \bigr\}
\end{equation}
where $U\in\ZZ^{m\times d}$ and $b\in\ZZ^m$, such that $m$ is minimal, i.e., each row $(u_i,-b_i)$ of the extended matrix $(U | -b)$ defines one facet of $P$, and $u_i$ is primitive.
These requirements make the extended matrix $(U|-b)$ unique up to reordering the rows.
Now the vector $\epsilon_i\in\{\pm 1\}^d$ is defined as follows: its $k$th coordinate is $-1$ raised to the power of the $k$th coordinate of $u_i$.
That vector is extended by one additional coordinate to $\epsilon_i^+=((-1)^{b_i},\epsilon_i)$ in $\{\pm1\}^{1+d}$.
We arrive at the following characterization.
\begin{theorem}[{\cite[Theorem 7.7]{Sottile:2011}}]
  \label{thm:orientable}
  The smooth locus of $Y_\cA^+$ is orientable if and only if there is a basis for $\{\pm1\}^{1+d}$ such that for every $i\in[m]$ the vector $\epsilon_i^+$ is a product of an odd number of basis vectors.
\end{theorem}

The following example already occurs as \cites[Example 3.3]{SoprunovaSottile:2006}[Example 7.10]{Sottile:2011}.
\begin{example}\label{exmp:hexagon}
  Let $\cA$ be the configuration of the seven lattice points \[ (0,0),\ (1,0),\ (0,1),\ (1,1),\ (1,2),\ (2,1),\ (2,2) \] in $\RR^2$.
  So $P=\conv(\cA)$ is a lattice hexagon.
  The height vector $\omega = (3, 1, 1, 0, 1, 1, 3)$ induces a foldable regular triangulation $T$ with signature 2, as seen in the figure below. 
  The vertices are labeled $0,1,2$, and this yields a proper coloring of the graph of $T$.
  
  \begin{center}
    \definecolor{col0}{rgb}{1.0,0.0,0.0}
    \definecolor{col1}{rgb}{0.0,1.0,0.0}
    \definecolor{col2}{rgb}{0.0,0.0,1.0}
    \definecolor{dark}{rgb}{0.4,0.4,0.4}
    \begin{tikzpicture}
      \draw[fill=dark] (1,0) -- (2,1) -- (1,1) -- cycle;
      \draw[fill=dark] (1,1) -- (1,2) -- (0,1) -- cycle;
      \draw[fill=dark] (0,0) -- (1,0) -- (0,1) -- cycle;
      \draw[fill=dark] (1,2) -- (2,2) -- (2,1) -- cycle;
      
      \node[circle, draw=black, fill=white, inner sep=0pt,minimum size=0.7em] at (0,0) {\tiny $0$};
      %\draw (0,0) node[left]{\small $3$};
      
      \node[circle, draw=black, fill=white, inner sep=0pt,minimum size=0.7em] at (1,1)  {\tiny $0$};
     % \draw (1,1) node[below left]{\small $0$};
      
      \node[circle, draw=black, fill=white, inner sep=0pt,minimum size=0.7em] at (1,0)  {\tiny $1$};
      %\draw (1,0) node[right]{\small $1$};
      
      \node[circle, draw=black, fill=white, inner sep=0pt,minimum size=0.7em] at (0,1)  {\tiny $2$};
      %\draw (0,1) node[left]{\small $1$};
      
      \node[circle, draw=black, fill=white, inner sep=0pt,minimum size=0.7em] at (2,2) {\tiny $0$};
      %\draw (2,2) node[right]{\small $3$};
      
      \node[circle, draw=black, fill=white, inner sep=0pt,minimum size=0.7em] at (2,1)  {\tiny $2$};
      %\draw (2,1) node[right]{\small $1$};
      
      \node[circle, draw=black, fill=white, inner sep=0pt,minimum size=0.7em] at (1,2)  {\tiny $1$};
      %\draw (1,2) node[left]{\small $1$};

    \end{tikzpicture}  
  \end{center}
  For parameters $c=(c_0,c_1,c_2)$ a Wronski polynomial $W_{\kappa,c}$ is now given by 
  \begin{equation}
    \label{wronski_polynomials}
    c_0(t^3 + x_1x_2 + t^3x_1^2x_2^2) + c_1t(x_1+ x_1x_2^2) + c_2t(x_2+x_1^2x_2) \enspace.
  \end{equation}
  The normalized volume of $P$ equals six, whence a system of two polynomials of the form \eqref{wronski_polynomials} with generic choices for $t, a_0, a_1, a_2 \in \CC$ will have six complex solutions.

  The orientability of $Y_\cA^+$ has been verified via Theorem~\ref{thm:orientable} in \cite[Example 7.8]{Sottile:2011}.
  To check if Theorem~\ref{thm:Soprunova+Sottile} applies, we consider the polynomial system \eqref{eq:metasys} in the
  three variables $t, x_1, x_2$ which reads:
  \begin{equation}
    \label{coefficient_system}
    \begin{cases}
      t^3 + x_1x_2 + t^3x_1^2x_2^2 = 0 \\
      t(x_1+ x_1x_2^2) = 0 \\
      t(x_2+x_1^2x_2) = 0
    \end{cases}
  \end{equation}
  It can be checked that there is no solution to the above system with $t \neq 0$.
  Consequently, $t^{-1}.Y_\cA \cap E_{\omega,\kappa}$ is empty for all nonzero values of $t$.
  So, by Theorem~\ref{thm:Soprunova+Sottile} any Wronski system \eqref{wronski_polynomials} with $c_i \in \RR\setminus\{0\}$ has at least two real solutions.
  
  Using the implementation of the symbolic library \msolve for \OSCAR \cite{oscar-book}, we computed the number of real solutions of 10,000 random instances of such a system. To be more precise, we chose $t$ to be uniformly distributed in $(-1,1)$, and the $a_i$'s to be uniformly distributed in $(-50,50)$. Of these systems, $79\%$ had 2 real solutions and the remaining $21\%$ had 6. None of the systems had 4 real solutions.
  Similar observations have been made in the experiments reported as \cite[Example 3.6]{SoprunovaSottile:2006}.
\end{example}

It turns out that, in the context of Theorem \ref{thm:Soprunova+Sottile}, a suitable height function always exists which
satisfies the condition \eqref{eq:extra} at least in the torus.

\begin{proposition}\label{prop:resultant}
  Let $T$ be a regular, dense and foldable triangulation of the lattice $d$-polytope $P\subset\RR^d$; we set $\cA=P\cap\ZZ^d$.
  Then there exists a height function $\omega: \cA \rightarrow \NN$ inducing $T$ such that the system
  \begin{equation}
    \label{eq:tosolve}
    f_0=\dots = f_d = 0
  \end{equation}
  with $f_i$ defined as in \eqref{eq:metasys} has no solutions in the torus $(\CC^{*})^d$ for all except finitely many values of $t$.
\end{proposition}
\begin{proof}
  Let $P_i$ be the convex hull of $\cA_i$ for $i=0,\dots,d$. Denote $\cA_i':=P_i\cap \ZZ^d$ and consider the sparse
  $(\cA_0',\dots,\cA_d')$-resultant
  \[R_1\in \ZZ[\gamma_{i,a}\;|\;i=0,\dots,d; a\in \cA_i\cap \ZZ^d] \enspace ;\]
  see \cite[see Chapter 8, Proposition-Definition 1.1]{gelfand1994} for a definition.
  Let $R$ be $R_1$ under the substitution $\gamma_{i,a}\gets 0$ for $a\in \cA_i'\setminus \cA_i$.
  Since each $\cA_i$ contains the set of vertices of each $P_i$, the polynomial $R$ is nonzero, by \cite[Theorem 1.17]{bender2024}.
  Because $\cA_i\cap \cA_j=\emptyset$ for $i\neq j$ we may consider $R$ as a polynomial in $\ZZ[\gamma_a\;|\; a\in \cA\cap \ZZ^d]$.

  Next, we have to show that there exists a height function $\omega:\cA\rightarrow \NN$ such that $R$ does not vanish under the substitution
  $\gamma_a\gets t^{\omega(a)}$. Let $M:=\ZZ^q$ where $q$ is the cardinality of $\cA\cap \ZZ^d$, we index the standard basis of
  $M$ by $\cA\cap \ZZ^d$. Each choice of $\omega$ induces a linear functional $\omega^{*}$ on $M$ by mapping the unit vector indexed
  by $a$ to $\omega(a)$, $\omega$ induces the triangulation $T$ if and only if $\omega^{*}$ lies in the secondary cone of
  $T$.  This cone is full-dimensional, thus there exists a choice of $\omega$ such that $\omega^{*}$ is injective when restricted to
  the exponent vectors of $R$, in particular $R$ does not vanish under the substitution $\gamma_a\gets t^{\omega(a)}$. Using the
  property of $R$ that a mixed system with supports $\cA_i$ has no solutions in $(\CC^{*})^d$ when the
  $R$ does not vanish, this shows that \eqref{eq:tosolve} has no solutions in $(\CC^{*})^d$.
\end{proof}

If the polytope $P$ and the triangulation $T$ in the setting of Proposition \ref{prop:resultant} satisfy a special local smoothness condition at the origin, then the conclusion of Proposition \ref{prop:resultant} can be extended to all of $\CC^d$.
\begin{theorem}\label{thm:lifting-exists}
  In the setting of Proposition \ref{prop:resultant} we additionally assume that $P\subset \RR_{\geq 0}^d$ is contained in the positive orthant
  and, moroever, the $d$-dimensional standard simplex $\Delta_d$ is a maximal cell of the triangulation $T$.
  Then the system \eqref{eq:tosolve} has no solutions in the entire space $\CC^d$ for almost all except finitely many values of $t$.
\end{theorem}
\begin{proof}
  The assumptions on $P$ and $T$ combined force that the origin is a vertex of $P$, and that the $d$ coordinate hyperplanes $H_i=\setof{x\in\RR^d}{x_i=0}$ define the facets of $P$ locally at the vertex $0$.
  In particular, $0$ is a simple vertex, i.e., the vertex figure of $P$ at $0$ is a $(d{-}1)$-dimensional simplex.
  
  Let $I\subseteq\{1,2,\dots,d\}$ be an arbitrary set of coordinate directions.  This defines a face
  $F=P\cap\bigcap_{i\in I}H_i$ of dimension is $d-k$ of $P$, where $k=\#I$.  Since $\Delta_d$ is a cell of $T$, the face
  $F$ of $P$ contains the face $G=\Delta_d\cap\bigcap_{i\in I}H_i$ of $\Delta_d$.  Let $J\subseteq\{0,1,\dots,d\}$ be the set of colors of the
  lattice points on the face $G$.  By construction we then have $\#J=d-k+1$.

  The above procedure defines a new system of $d-k+1$ polynomial equations
  \begin{equation}\label{eq:local-meta}
    (f_j)^I=0 \quad \text{for all } j\in J \enspace,
  \end{equation}
  in $d-k$ variables, here $(f_j)^I$ is the substitution of $f_j$ by $x_i\gets 0$ for all $i\in I$.
  The polynomials $(f_j)^I$ are nonzero, because for each color $j\in J$ there is at least one lattice point in $F$ of color $j$.

  For $j\in J$ let $P_j=\conv(\cA_j)$ be the convex hull of the lattice points of color $j$.
  In general, $P_j$ is neither a face nor a subpolytope of $P$.
  However, the set $F\cap P_j$ forms a nonempty face of $P_j$, of some unknown dimension.
  We infer that each vertex of $F\cap P_j$ is also a vertex of $P_j$.

  So we can construct a nonzero polynomial $R_I\in\ZZ[\gamma_a|a\in\cA\cap F]$, as in the proof of Theorem \ref{thm:lifting-exists}.
  Again $R_I$ has the property that a system with the same support as \eqref{eq:local-meta} has no zero in the subtorus associated to the complement of $I$.
  We may now find, as in the proof of Theorem \ref{thm:lifting-exists}, a height function $\omega$ such that the linear function $\omega^*$ is injective on the union of all supports of $R_I$ with varying $I$.
  Hence the system \eqref{eq:tosolve} has no solutions in the union
  \[\CC^d = \bigcup_{I\subset \{1,\dots,d\}} \prod_{i\notin I} \CC^* \times \prod_{i\in I}\{0\} \]
  of subtori for almost all values of $t$.
  This completes our proof.
\end{proof}

\section{Honeycomb triangles and Wronski curves with full support} \label{sec:alcoved}
To set the context, in out next step we briefly introduce a class of polytopes which occurs frequently in algebraic combinatorics and elsewhere.
This information will then be used to specialize to the planar setting, in order to study Wronski curves of a given degree with full support.

\subsection*{Alcoved polytopes.}
For the general context, we start out with a construction studied by Lam and Postnikov \cite{Lam+Postnikov:2007}.
The affine Coxeter arrangement of type $A_{n-1}$ is the infinite arrangement of hyperplanes in $\RR^{n-1}$ given by
\begin{equation}
  \label{eq:arrangement:A}
  H_{ij}^\ell \ = \ \bigl\{ (z_1 , \ldots , z_{n-1}) \in\RR^{n-1} \bigm| z_i-z_j =\ell \bigr\} \enspace ,
\end{equation}
where $0\leq i < j \leq n-1$, $\ell\in\ZZ$, and $z_0=0$.
An \emph{alcoved polytope} is a polytope, $P$, which admits an inequality description in terms of finitely many hyperplanes $H_{ij}^\ell$.
Equivalently, there are integer vectors $b,c$ of length $\tbinom{n}{2}$, with $b_{ij}\leq c_{ij}$, such that
\begin{equation}
  \label{eq:alcoved}
  P \ = \ P(b,c) \ = \ \bigl\{ (z_1 , \ldots , z_{n-1}) \in\RR^{n-1} \bigm| b_{ij} \leq z_i-z_j \leq c_{ij} \bigr\} \enspace .
\end{equation}
The hyperplane arrangement \eqref{eq:arrangement:A} induces a triangulation of $\RR^{n-1}$ into unimodular simplices, i.e., simplices of normalized volume one; such a triangulation is called \emph{unimodular}.
Restricting that triangulation to the alcoved polytope $P$ gives a regular, dense and foldable triangulation of the point configuration $\cA=P\cap\ZZ^{n-1}$.
That triangulation is the \emph{alcoved triangulation} of $\cA$.
One possible height function valid for the entire alcoved triangulation of $\ZZ^{n-1}$ is, e.g.,
\begin{equation}\label{eq:lifting}
  \omega(z_1, \dots, z_{n-1}) \ = \ \sum_{i=0}^{n-1} z_i^2 + \sum_{\substack{i,j \in \{1, \dots, n-1 \} \\ i<j}}^{n-1} (z_i-z_j)^2 \enspace ,
\end{equation}
which occurs, e.g., in \cite{sinn2021}.

\begin{example}
  The hexagon from Example~\ref{exmp:hexagon} is a first example of alcoved polytope, albeit with a different triangulation than the one given in Example \ref{exmp:hexagon}.
  Its defining hyperplanes are $H_{01}^{-2}, H_{02}^{-2}, H_{12}^{-1}, H_{01}^0, H_{02}^0, H_{12}^{1}$.
  So that hexagon is equal to $P(b,c)$ where $b = (-2, -2,-1)$ and $c = (0,0,1)$. The height function in \ref{eq:lifting}, restricted to the hexagon, gives the height vector $(0,1,1,2,0,0,1)$. 

  \begin{center}
    \definecolor{col0}{rgb}{1.0,0.0,0.0}
    \definecolor{col1}{rgb}{0.0,1.0,0.0}
    \definecolor{col2}{rgb}{0.0,0.0,1.0}
    \definecolor{dark}{rgb}{0.4,0.4,0.4}
    \begin{tikzpicture}
      \draw[fill=dark] (0,0) -- (1,0) -- (1,1) -- cycle;
      \draw[fill=dark] (1,1) -- (2,1) -- (2,2) -- cycle;
      \draw[fill=dark] (0,1) -- (1,1) -- (1,2) -- cycle;
      \draw (1,0) -- (2,1);
      \draw (0,0) -- (0,1);
      \draw (1,2) -- (2,2);
      
      	\node[circle, draw=black, fill=white, inner sep=0pt,minimum size=0.7em] at (0,0) {\tiny $0$};
%\draw (0,0) node[left]{\small $3$};

\node[circle, draw=black, fill=white, inner sep=0pt,minimum size=0.7em] at (1,1)  {\tiny $2$};
% \draw (1,1) node[below left]{\small $0$};

\node[circle, draw=black, fill=white, inner sep=0pt,minimum size=0.7em] at (1,0)  {\tiny $1$};
%\draw (1,0) node[right]{\small $1$};

\node[circle, draw=black, fill=white, inner sep=0pt,minimum size=0.7em] at (0,1)  {\tiny $1$};
%\draw (0,1) node[left]{\small $1$};

\node[circle, draw=black, fill=white, inner sep=0pt,minimum size=0.7em] at (2,2) {\tiny $1$};
%\draw (2,2) node[right]{\small $3$};

\node[circle, draw=black, fill=white, inner sep=0pt,minimum size=0.7em] at (2,1)  {\tiny $0$};
%\draw (2,1) node[right]{\small $1$};

\node[circle, draw=black, fill=white, inner sep=0pt,minimum size=0.7em] at (1,2)  {\tiny $0$};
%\draw (1,2) node[left]{\small $1$};

    \end{tikzpicture}  
  \end{center}

\end{example}

Hypersimplices form an entire family of examples; see \cite[Eq.~(1)]{Lam+Postnikov:2007}.
The order polytopes studied in \cite[\S8.1]{Sottile:2011} are alcoved, up to a linear transformation \cite[\S7.2]{Lam+Postnikov:2007}.

\subsection*{Honeycomb triangles.}
The standard simplex $\Delta_d=\conv\{0,e_1,e_2,\dots,e_d\}$ in $\RR^d$ is never an alcoved polytope in the sense of \eqref{eq:alcoved}.
The reason is that the facet defining inequality $x_1+x_2+\dots+x_d\leq 1$ does not belong to the Coxeter hyperplane arrangement \eqref{eq:arrangement:A}.
However, in the plane this can be remedied in the following way.

The triangle $\Delta_2'=\conv\{0,e_1,e_1+e_2\}$ is alcoved, and the unimodular transformation
\begin{equation}\label{eq:tau}
  \tau \ = \ \begin{pmatrix} 1 & -1 \\ 0 & 1\end{pmatrix} \quad \in \, \SL(2,\ZZ)
\end{equation}
satisfies $\Delta_2=\tau\cdot\Delta_2'$.
So, in this sense, we could view the dilated standard triangle $\delta\Delta_2=\conv\{(0,0),(\delta,0),(0,\delta)\}$ as an alcoved polygon.
Yet, since coordinates matter here we want to distinguish the two.
Therefore, as in \cite{BJMS:2015}, we call the image of the alcoved triangulation of $\delta\Delta_2'$ under $\tau$ the \emph{honeycomb triangulation} of $\delta\Delta_2$.
The lattice points in $\delta\Delta_2$ correspond bijectively to the bivariate monomials of degree at most $\delta$.
In other words, the honeycomb triangle $\delta\Delta_2$ is the Newton polygon of a bivariate polynomial of degree $\delta$ with full support.
The case $\delta=3$ occurs as \cite[Example~5.1]{SoprunovaSottile:2006}; see Figure~\ref{fig:delta=3}.

With any height function satisfying the conditions of Proposition \ref{prop:liftfunc}, a Wronski polynomial for
$\delta\Delta_2$ is the following:
\begin{equation}
	\label{eq:wronski-triangle}
	c_0(\sum_{(i,j) \in \cA_0} t^{\omega(i,j)} x^iy^j) + c_1(\sum_{(i,j) \in \cA_1} t^{\omega(i,j)} x^iy^j) + c_2(\sum_{(i,j) \in \cA_2} t^{\omega(i,j)} x^iy^j)
\end{equation}
where the $c_k$, $k=0,1,2$, correspond to the colors and $\cA_i$ are the sets of lattice points of the corresponding colors, defined by 
\[\cA_k=\{(i,j) \in \delta\Delta_2 \mid i-j \equiv k \bmod 3\};\; k = 0,1,2\]
as seen in Figure~\ref{fig:delta=3}.

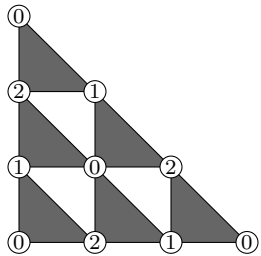
\begin{figure}[th]
  \centering
  \definecolor{dark}{rgb}{0.4,0.4,0.4}
  \begin{tikzpicture}
    \draw[fill=dark] (1,1) -- (0,1) -- (0,2) -- cycle;
    \draw[fill=dark] (1,1) -- (1,0) -- (2,0) -- cycle;
    \draw[fill=dark] (0,0) -- (1,0) -- (0,1) -- cycle;
    \draw[fill=dark] (1,1) -- (1,2) -- (2,1) -- cycle;
    \draw[fill=dark] (2,0) -- (2,1) -- (3,0) -- cycle;
    \draw[fill=dark] (0,2) -- (1,2) -- (0,3) -- cycle;
    
    \node[circle, draw=black, fill=white, inner sep=0pt,minimum size=0.7em] at (0,0) {\tiny $0$};
    \node[circle, draw=black, fill=white, inner sep=0pt,minimum size=0.7em] at (1,1) {\tiny $0$};
    \node[circle, draw=black, fill=white, inner sep=0pt,minimum size=0.7em] at (1,0) {\tiny $2$};
    \node[circle, draw=black, fill=white, inner sep=0pt,minimum size=0.7em] at (0,1) {\tiny $1$};
    \node[circle, draw=black, fill=white, inner sep=0pt,minimum size=0.7em] at (2,1) {\tiny $2$};
    \node[circle, draw=black, fill=white, inner sep=0pt,minimum size=0.7em] at (1,2) {\tiny $1$};
    \node[circle, draw=black, fill=white, inner sep=0pt,minimum size=0.7em] at (0,2) {\tiny $2$};
    \node[circle, draw=black, fill=white, inner sep=0pt,minimum size=0.7em] at (2,0) {\tiny $1$};
    \node[circle, draw=black, fill=white, inner sep=0pt,minimum size=0.7em] at (0,3) {\tiny $0$};
    \node[circle, draw=black, fill=white, inner sep=0pt,minimum size=0.7em] at (3,0) {\tiny $0$};
    
  \end{tikzpicture}
  \caption{Honeycomb triangulation of $3\Delta_2$}
  \label{fig:delta=3}
\end{figure}

\begin{proposition}
  \label{prop:orientability-triangle}
  Let $\cA=\delta\Delta_2\cap\ZZ^2$ be the set of lattice points in the dilated standard triangle.
  Then the smooth locus of $Y_\cA^+$ is orientable if and only if $\delta$ is odd.
\end{proposition}
\begin{proof}
  Instead of $\delta\Delta_2$ we analyze the lattice equivalent triangle $\delta\Delta_2'$.
  Its three vertices are $0$, $e_1$ and $e_1+e_2$, and its three facets read $-x \geq -\delta$, $y \geq 0$ and $-x+y \geq 0$.
  From this we compute 
  \[
    \begin{aligned}
      \epsilon_1^+ &= ((-1)^{\delta},-1,1) \\
      \epsilon_2^+ &= (1,1,-1) \\
      \epsilon_3^+ &= (1,-1,-1) 
    \end{aligned}
  \]
  For $\delta$ odd these three vectors form a basis.
  However, if $\delta$ is even, then the first column of the $3{\times}3$-matrix with rows $\epsilon_i^+$ is trivial.
  Hence the row rank (over GF$(2)$ in exponential form) equals two.
  The claim follows from Theorem~\ref{thm:orientable}.
\end{proof}

% \todo[inline]{%
%   * Can you prove that the meta-condition $t^{-1}.Y_\cA \cap E_{\omega,\kappa}=0$ holds for Wronski systems of $\delta\Delta_2$?
%   What about $\delta\Delta_{d-1}$ for arbitrary $d$?
% }

% \begin{theorem}
%   For odd $\delta \leq 9$ each Wronski polynomial system with respect to $\delta\Delta_2$ has at least $\delta$ many real solutions. \todo{$\delta$ even}
% \end{theorem}
% \begin{proof}
%   computational
% \end{proof}

The main result of \cite[Theorem 3]{JoswigZiegler:2014} is a general formula for signatures of triangulations of lattice polygons.
The honeycomb triangles are explicitly mentioned in loc.\ cit.
\begin{proposition}[{\cite[Example 7]{JoswigZiegler:2014}}]
  The honeycomb triangulation of the dilated standard triangle $\delta\Delta_2$ has signature $\delta$.
\end{proposition}

It is straightforward to describe the f-vector of the honeycomb triangulation.
A unimodular triangulation of a lattice polytope is necessarily full, in arbitrary dimension.
In the plane the converse is true.
Therefore, for the rest of this paper, we can use the terms \enquote{unimodular} and \enquote{full} interchangeably.

\begin{lemma}\label{lem:fvector}
  Every full triangulation of $\delta\Delta_2$ has exactly $\tfrac{1}{2}(\delta+1)(\delta+2)$ vertices, $\tfrac{3}{2}(\delta^2+\delta)$ edges and $\delta^2$ triangles.
  There are precisely $\tfrac{1}{2}(\delta-2)(\delta-1)$ interior vertices and $\tfrac{3}{2}(\delta^2-\delta)$ interior edges.
\end{lemma}
\begin{proof}
  The number of vertices is clear.
  By unimodularity the number of triangles agrees with the normalized area.
  With this information equipped, the number of edges follows because the Euler characteristic of a ball equals one.
  There are exactly $3\delta$ boundary edges; this determines the number of interior edges.
\end{proof}

Note that the number of interior lattice points of the triangle $\delta\Delta_2$ is exactly the geometric genus of the associated Wronski curves, which are smooth for a unimodular triangulation.
To study the dependence on the height function we provide a full characterization.
That is, we give an irredudant inequality description of the secondary cone of the honeycomb triangulation of $\delta\Delta_2$, including its $3$-dimensional lineality space.

\begin{proposition}
  \label{prop:liftfunc}
  A function $\omega:\delta\Delta_2\cap\ZZ^2\rightarrow \RR$ induces the honeycomb triangulation on $\delta\Delta_2$ if and only if the following strict homogeneous linear inequalities are satisfied, for all integral $i,j\geq 1$ with $i+j\leq\delta$:
  \begin{enumerate}
  \item $\omega(i-1,j-1)+\omega(i,j) > \omega(i-1,j)+\omega(i,j-1)$, 
  \item $\omega(i-1,j)+\omega(i+1,j-1) > \omega(i,j-1)+\omega(i,j)$, and
  \item $\omega(i-1,j-1)+\omega(i-1,j+1) > \omega(i-1,j)+\omega(i,j)$.
  \end{enumerate}
\end{proposition}
\begin{proof}
  According to \cite[Corollary~5.2.7]{Triangulations} the facets of the secondary cone of any triangulation bijectively correspond to the internal cells of codimension one.
  In Lemma~\ref{lem:fvector} we saw that the honeycomb triangulation has exactly $\tfrac{3}{2}(\delta^2-\delta)$ interior edges, which are codimension one in the plane.
  The given inequalities express the conditions on the \enquote{folding form} from \cite[Definition~5.2.4]{Triangulations}.
\end{proof}
Since the facet description of the secondary cone does not mention the scaling parameter $\delta$, the following consequence is immediate.

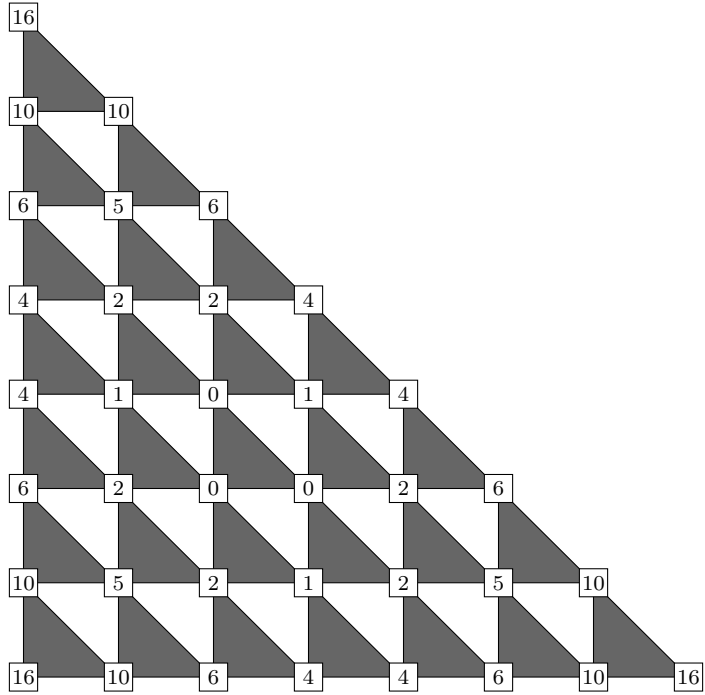
\begin{figure}[th]
  % polymake for joswig
% Wed Nov 13 16:18:30 2024
% pcom:S0

  \begin{tikzpicture}[scale=1.25]
	\definecolor{dark}{rgb}{0.4,0.4,0.4}
	\foreach \x in {0,1,...,6}{
		\pgfmathtruncatemacro{\y}{7-\x}
		\foreach \k in {0,...,\x}{
			\draw[fill=dark] (\x,\y) (\k,\y-1) -- (\k+1,\y-1) -- (\k,\y) -- cycle;}
	}

\node[rectangle, draw=black, fill=white, inner sep=0pt,minimum size=0.9em] at (0,0) {\tiny $16$};
\node[rectangle, draw=black, fill=white, inner sep=0pt,minimum size=0.9em] at (0,1) {\tiny $10$};
\node[rectangle, draw=black, fill=white, inner sep=0pt,minimum size=0.9em] at (0,2) {\tiny $6$};
\node[rectangle, draw=black, fill=white, inner sep=0pt,minimum size=0.9em] at (0,3) {\tiny $4$};
\node[rectangle, draw=black, fill=white, inner sep=0pt,minimum size=0.9em] at (0,4) {\tiny $4$};
\node[rectangle, draw=black, fill=white, inner sep=0pt,minimum size=0.9em] at (0,5) {\tiny $6$};
\node[rectangle, draw=black, fill=white, inner sep=0pt,minimum size=0.9em] at (0,6) {\tiny $10$};
\node[rectangle, draw=black, fill=white, inner sep=0pt,minimum size=0.9em] at (0,7) {\tiny $16$};

\node[rectangle, draw=black, fill=white, inner sep=0pt,minimum size=0.9em] at (1,0) {\tiny $10$};
\node[rectangle, draw=black, fill=white, inner sep=0pt,minimum size=0.9em] at (2,0) {\tiny $6$};
\node[rectangle, draw=black, fill=white, inner sep=0pt,minimum size=0.9em] at (3,0) {\tiny $4$};
\node[rectangle, draw=black, fill=white, inner sep=0pt,minimum size=0.9em] at (4,0) {\tiny $4$};
\node[rectangle, draw=black, fill=white, inner sep=0pt,minimum size=0.9em] at (5,0) {\tiny $6$};
\node[rectangle, draw=black, fill=white, inner sep=0pt,minimum size=0.9em] at (6,0) {\tiny $10$};
\node[rectangle, draw=black, fill=white, inner sep=0pt,minimum size=0.9em] at (7,0) {\tiny $16$};

\node[rectangle, draw=black, fill=white, inner sep=0pt,minimum size=0.9em] at (1,1) {\tiny $5$};
\node[rectangle, draw=black, fill=white, inner sep=0pt,minimum size=0.9em] at (1,2) {\tiny $2$};
\node[rectangle, draw=black, fill=white, inner sep=0pt,minimum size=0.9em] at (1,3) {\tiny $1$};
\node[rectangle, draw=black, fill=white, inner sep=0pt,minimum size=0.9em] at (1,4) {\tiny $2$};
\node[rectangle, draw=black, fill=white, inner sep=0pt,minimum size=0.9em] at (1,5) {\tiny $5$};
\node[rectangle, draw=black, fill=white, inner sep=0pt,minimum size=0.9em] at (1,6) {\tiny $10$};

\node[rectangle, draw=black, fill=white, inner sep=0pt,minimum size=0.9em] at (2,1) {\tiny $2$};
\node[rectangle, draw=black, fill=white, inner sep=0pt,minimum size=0.9em] at (3,1) {\tiny $1$};
\node[rectangle, draw=black, fill=white, inner sep=0pt,minimum size=0.9em] at (4,1) {\tiny $2$};
\node[rectangle, draw=black, fill=white, inner sep=0pt,minimum size=0.9em] at (5,1) {\tiny $5$};
\node[rectangle, draw=black, fill=white, inner sep=0pt,minimum size=0.9em] at (6,1) {\tiny $10$};

\node[rectangle, draw=black, fill=white, inner sep=0pt,minimum size=0.9em] at (2,2) {\tiny $0$};
\node[rectangle, draw=black, fill=white, inner sep=0pt,minimum size=0.9em] at (2,3) {\tiny $0$};
\node[rectangle, draw=black, fill=white, inner sep=0pt,minimum size=0.9em] at (2,4) {\tiny $2$};
\node[rectangle, draw=black, fill=white, inner sep=0pt,minimum size=0.9em] at (2,5) {\tiny $6$};

\node[rectangle, draw=black, fill=white, inner sep=0pt,minimum size=0.9em] at (3,2) {\tiny $0$};
\node[rectangle, draw=black, fill=white, inner sep=0pt,minimum size=0.9em] at (4,2) {\tiny $2$};
\node[rectangle, draw=black, fill=white, inner sep=0pt,minimum size=0.9em] at (5,2) {\tiny $6$};

\node[rectangle, draw=black, fill=white, inner sep=0pt,minimum size=0.9em] at (3,3) {\tiny $1$};
\node[rectangle, draw=black, fill=white, inner sep=0pt,minimum size=0.9em] at (3,4) {\tiny $4$};
\node[rectangle, draw=black, fill=white, inner sep=0pt,minimum size=0.9em] at (4,3) {\tiny $4$};
\end{tikzpicture}
  \caption{The minimal height function $\mu$ for $\delta=7$.
    As $7=3\cdot 3+1$ one can observe two inductive steps of the procedure sketched in Remark~\ref{rem:minimal}.
    This procedure starts at the central triangle $(2,2)+\Delta_2$, passes through $(1,1)+4\Delta_2$, and finally arrives at $7\Delta_2$.}
  \label{fig:minimal}
\end{figure}

\begin{corollary}
  A function $\omega:\ZZ^2\rightarrow \RR$ induces the honeycomb triangulation on the entire plane $\RR^2$ if and only if the three classes of strict homogeneous linear inequalities from Proposition~\ref{prop:liftfunc} are satisfied, for all $i,j\in\ZZ$.
\end{corollary}

\begin{remark}\label{rem:minimal}
  The secondary cone of any regular triangulation, $T$, of the point configuration $\cA$ is a full-dimensional open rational polyhedral cone in $\RR^\cA$.
  So any height function $\omega$ inducing $T$ on $\cA$ satisfies strict homogeneous linear inequalities like $\sum_{p\in\cA} u_p \omega_p > 0$ with $u_p\in\ZZ$.
  If $\omega$ is integral, then it satisfies the weak inhomogeneous inequality $\sum{p\in\cA} s_p \omega_p \geq 1$.
  Consequently, for given $T$, hence a nonnegative integral height vector $\omega$ which minimizes, e.g., $\sum_{p\in\cA} \omega_p$ can always be found by solving an integer linear program.

  If $T$ is the honeycomb triangulations of $\delta\Delta_2$, then such a minimal height function can be described inductively.
  For conciseness we assume that $\delta$ is congruent one modulo three.
  Then we can argue by induction; in the base case $\delta=1$ we take $\omega=0$, which is clearly optimal.
  Now we make the inductive step from $\delta$ to $\delta+3$.
  The triangle $P=(\delta+3)\Delta_2$ contains the triangle $P'=\conv((1,1),(\delta-1,1),(1,\delta-1))$ which is the same as $\delta\Delta_2$ shifted by the vector $(1,1)$.
  By induction we can assume that omega is defined on the lattice points $P'\cap\ZZ^2$.
  The remaining lattice points $p\in(P\setminus P')\cap\ZZ^2$ receive heights according to a distinction of three cases.
  In the first case, we assume that $p$ is the unique vertex of the triangle $p,q,r$ in $T$ which shares the edge $q,r$ with $P'$.
  Then there is a unique lattice point $s\in P'\cap\ZZ^2$ such that $q,r,s$ is a triangle in $T$.
  We can set
  \begin{equation}\label{eq:minimal}
    \omega(p) \ = \ \omega(q)+\omega(r)-\omega(s)+1 \enspace,
  \end{equation}
  where $\omega(q),\omega(r),\omega(s)$ are known, integral and minimal by induction.
  In the second case, we consider those lattice points in the boundary of $P$ which lie on a boundary edge $p,q$ of $T$ with one of the vertices which received a height in the first case.
  There is a unique vertex $r$ such that $p,q,r$ is a triangle of $T$, which is again adjacent to a triangle $q,r,s$ which lies in $P'$.
  Applying \eqref{eq:minimal} again yields the height of $p$.
  Now all lattice points in $P$ have their heights, except for the three vertices of the triangle $P$.
  So in the third an final case $p$ is one of these three vertices.
  Then $p$ lies in a unique triangle $p,q,r$ of $T$, and again there is a fourth point $s$ to form two triangles which allow to apply the formula \eqref{eq:minimal}.
  Clearly in each case the choice was minimal.

  If $\delta$ is congruent one modulo three, the minimal height function is unique; this is not the case otherwise.
  Figure~\ref{fig:minimal} displays the minimal height function for $\delta=7$.
\end{remark}

\section{Computational Results} \label{sec:compute}

In this section we finally study explicit Wronski pairs of curves arising from the honeycomb traingulation of $\delta\Delta_2$.
For a choice of generic parameters and height function each such \emph{honeycomb curve} is smooth of degree $\delta$ with genus $\tfrac{1}{2}(\delta-1)(\delta-2)$; the latter number agrees with the number of interior lattice points of $\delta\Delta_2$; cf.\ Lemma~\ref{lem:fvector}.
Most importantly, the honeycomb triangulations satisfy the assumptions of Theorem \ref{thm:lifting-exists}.
Consequently, for each $\delta\geq 1$ there exists a height function such that any corresponding Wronski pair intersects in at least $\delta$ real points when $\delta$ is odd.
Note that, trivially, this number $\delta$ agrees with the square root of the product of the degrees.
In view of a celebrated result of Shub and Smale \cite{Shub+Smale:1993} this means that Wronski pairs of generic honeycomb curves of odd degree always attain or exceed the average number of real solutions.

A key goal of this section is to explore the frontiers of current computational methods when choosing a
\textit{concrete} height function.  More precisely, for a given height function $\omega:\delta\Delta_2\cap\ZZ^2\to\NN$ we want to use these
software tools to investigate the real solutions of the system
\begin{equation}
  \label{eq:tosolve2}
  \sum_{(i,j) \in \cA_0} t^{\omega(i,j)} x^iy^j \ = \ \sum_{(i,j) \in \cA_1} t^{\omega(i,j)} x^iy^j \ = \ \sum_{(i,j) \in \cA_2} t^{\omega(i,j)} x^iy^j \ = \ 0 \enspace.
\end{equation}
which corresponds to the extra condition in Theorem \ref{thm:Soprunova+Sottile}.
From this we may then verify, with some degree of certainty, that a minimal $t_0$ such that \eqref{eq:tosolve2} has no solutions for $0 < t < t_0$ exists.
Consequently, using Theorem \ref{thm:Soprunova+Sottile}, we may then explicitly construct Wronski pairs associated to the honeycomb triangulations of odd degree $\delta$ with at least $\delta$ real solutions.

Since it is known that the choice of the height function matters (in terms of the structure of the solutions of
\eqref{eq:tosolve2}), we tried two different examples. Our first choice is
\begin{equation}\label{eq:rainer}
  \rho(i,j) \ = \ i^2 + j^2 + ij \enspace,
\end{equation}
which arises from applying the transformation $\tau$ from \eqref{eq:tau} to the height function \eqref{eq:lifting}.  Our
second choice is the minimal height function from Remark~\ref{rem:minimal}; here we call it $\mu$.
This minimal height function $\mu$ has the benefit of minimizing the degrees in \eqref{eq:tosolve2}, which suggests that the computations might be easier.
However, we will see below that this is not necessarily the case.

For our investigation we used the numerical software \HC by Breiding and Timme \cite{breiding2018} and the symbolic library \msolve by Berthomieu, Eder and Safey El Din \cite{berthomieu2021}.

As the name suggests \HC uses homotopy continuation to solve polynomial systems. Roughly this means that one solves
first an \enquote{easy} system of a similar structure as the one one is interested in and then deforms this start system
to the target system, tracking the solutions along the way. This is a very efficient way to solve polynomial systems but
may suffer from issues of instability or inexaustiveness.
We refer to \cite{sommese2005} for an introduction. Some
guarantees however can be provided: \HC implements the certification method in \cite{breiding2023} which enables one to
attempt to construct a bounding intervall for each computed approximation of a solution if the polynomial system is square
and the solution is non-singular (i.e. the jacobian of the system does not vanish at this solution).

In contrast \msolve uses Gröbner basis methods to compute solutions to polynomial system. This is done by first
computing a Gröbner basis for the ideal generated by the system in question using the F4 algorithm \cite{faugere1999}
and then converting this Gröbner basis into a univariate parametrization of the system using variants of the FGLM
algorithm \cite{faugere2017a}. As a symbolic method, \msolve offers a great deal of reliability but is often slower
than a numerical approach. Strictly speaking, the results produced are probabilistic as \msolve uses multi-modular methods
for the underlying Gröbner basis computations.

% Before we proceed with our investigation let us illustrate the case of even $\delta$. For even $\delta$, according to Proposition
% \ref{prop:orientability-triangle}, the smooth locus of $Y_{\delta \Delta_2}^+$ is not orientable so we cannot apply Theorem
% \ref{thm:Soprunova+Sottile}. 

We start with the results obtained using \HC, recorded in Table~\ref{tbl:hc}.  In this table we record for both height
functions $\mu$ and $\rho$ respectively the number of certified real solutions, the number of singular (and thus
non-certifiable) real solutions and the time the computation took. For $\delta=17$ and $\omega = \mu$ \HC gave an error during the
construction of the start system. Note that even when all real solutions are certified, these results to not enable us
to extract a proven lower bound $t_0$ (in the context of Theorem \ref{thm:Soprunova+Sottile}) or guarantee that such a
lower bound exists because we cannot guarantee to have obtained all isolated solutions using \HC or that
positive-dimensional solutions are not present. Nonetheless, we can obtain some further evidence that we have obtained
the minimal $t_0$ in some cases, illustrated by the following examples for $\delta=3,5$ and $\omega = \rho$ for which the associated
systems of Wronski polynomials have no real solutions:

\colorlet{red1}{white!72!red}
\colorlet{red2}{white!63!red}
\colorlet{red3}{white!54!red}
\colorlet{red4}{white!45!red}
\colorlet{red5}{white!36!red}
\colorlet{red6}{white!27!red}
\colorlet{red7}{white!18!red}
\colorlet{red8}{white!9!red}

\colorlet{blue1}{white!72!blue}
\colorlet{blue2}{white!63!blue}
\colorlet{blue3}{white!54!blue}
\colorlet{blue4}{white!45!blue}
\colorlet{blue5}{white!36!blue}
\colorlet{blue6}{white!27!blue}
\colorlet{blue7}{white!18!blue}
\colorlet{blue8}{white!9!blue}

\begin{example}
  For $\delta = 3$ and $\omega = \rho$ we found no real solutions using \HC. For $\delta=5$ and $\omega = \rho$, the minimal
  $t_0>0$ for which \HC found a real solution was $1$. This would indicate that by Proposition
  \ref{prop:orientability-triangle} and Theorem \ref{thm:Soprunova+Sottile}, the corresponding Wronski systems
  arising from \ref{eq:wronski-triangle} have at least $3$ resp. $5$ real solutions, for generic choices of coefficients
  $c_0,c_1,c_2$ and $c_0',c_1',c_2'$ and any $t$ resp. $t<1$.

\begin{figure}[th] \centering
  \includegraphics[height=0.45\textwidth]{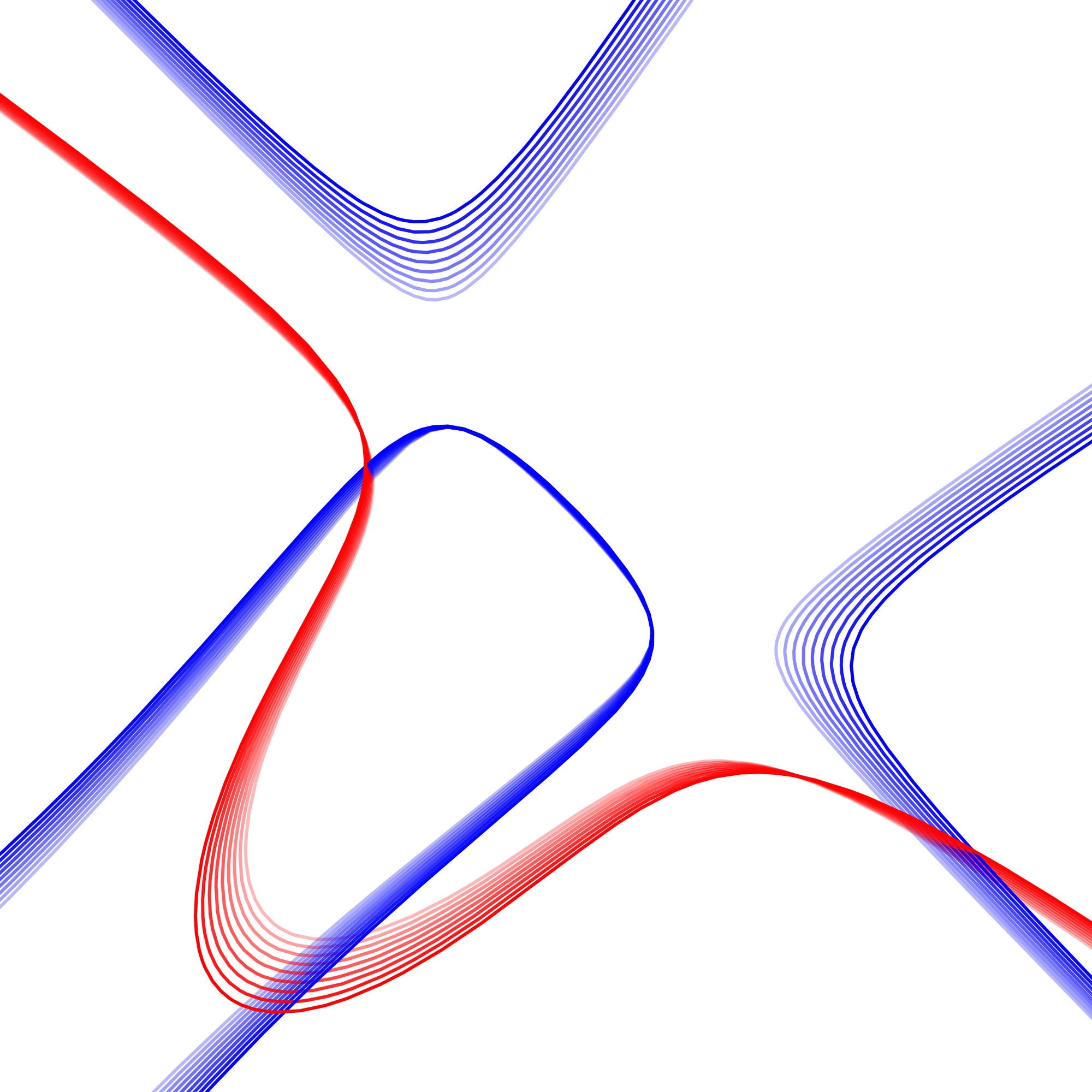}
  \includegraphics[height=0.45\textwidth]{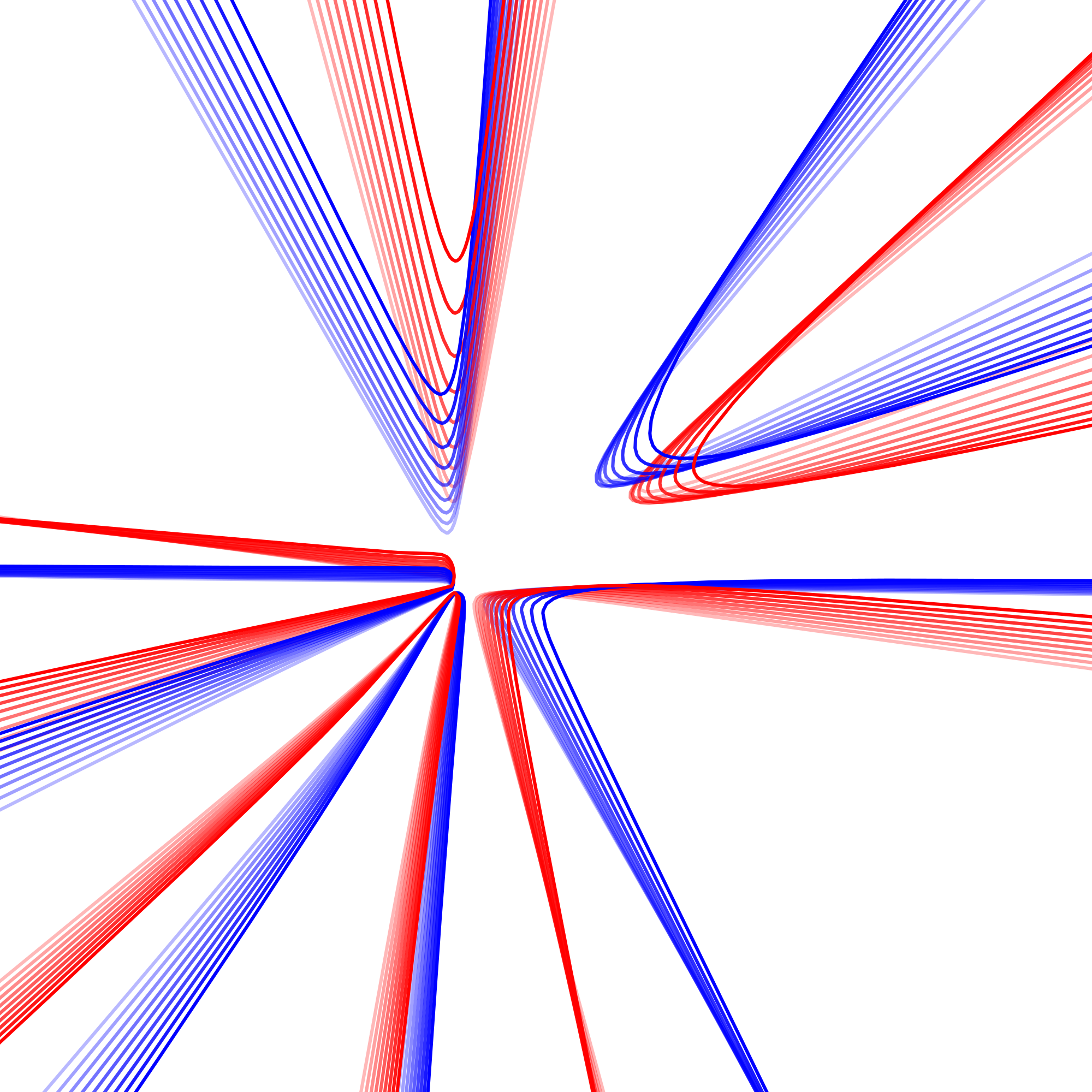}
  
  \caption{Left: family of Wronski pairs of cubic honeycomb curves intersecting in 3 points each.
    Right: family of Wronski pairs of quintic honeycomb curves intersecting in 5 points each.
    Each pair comprises one red and one blue curve.}
  \label{fig:rhocubic}
\end{figure}	        
  In Figure \ref{fig:rhocubic}, we depict two pairs of
  families of such Wronski curves with exactly $\delta$ points of intersection. On the left, $\delta$ is equal to $3$, the
  parameters are $(c_0,c_1,c_2)=(-3.14,-8.13,3.61)$, $(c_0',c_1',c_2')=(11.13,-9.34,1.82)$ and $t$ varies between $0.96$
  and $1$. On the right, $\delta$ is $5$, the parameters are $(c_0,c_1,c_2)=(0.79,0.11,-0.72)$,
  $(c_0',c_1',c_2')=(0.37,0.84,-0.97)$ and $t$ ranges from $0.52$ to $0.66$. We note that for random $t>1$,
  $\delta = 5$ and randomly chosen coefficients of the corresponding Wronski systems we only found systems with one real
  solution.
\end{example}

\begin{table*}[bh]
  \caption{Results using \HC}
  \label{tbl:hc}
  \begin{center}
    \begin{tabular}{l|lll|lll}
      $\omega$ & \multicolumn{3}{c}{$\rho$} & \multicolumn{3}{c}{$\mu$}\\
      $\delta$  & \# cert. real & \# sing. real & time & \# cert. real & \# sing. real & time\\
      \midrule
      3 & 0 & 0 & 2.1  & 0 & 1 & 2.3 \\   
      5 & 4 & 0 & 4.1 & 4 & 1 & 4.2 \\   
      7 & 4 & 0 & 5.8  & 4 & 1 & 5.2 \\   
      9 & 4 & 0 & 11.8  & 4 & 1 & 8.5  \\  
      11& 24 & 2 & 27.1  & 24 & 2 & 21.0 \\
      13& 10 & 2 & 59.3  & 10 & 2 & 42.4 \\
      15& 10 & 2 & 160.8 & 10 & 2 & 96.6 \\
      17 & 17 & 4 & 281.5 & \multicolumn{3}{c}{error}\\
    \end{tabular}
  \end{center}
\end{table*}

Next, we also attempted to compute the solutions to \eqref{eq:tosolve2} using both $\mu$ and $\rho$ with \msolve.  This 
returns the correct result to a guaranteed precision probabilistically. The data we obtained is summarized in Table
\ref{tbl:msolve}. For the minimal height function $\mu$, \msolve reported that the ideal $I_{\delta}$ generated by the
equations in \eqref{eq:tosolve2} is positive-dimensional, this is done probabilistically by computing the dimension
modulo a random prime number. This means in particular that \HC did not compute all solutions for the corresponding
cases. For $\delta = 9$ the computation did not finish after more than 12 hours. Otherwise the results are consistent with
the ones given in Table \ref{tbl:hc}.

\begin{table*}[bh]
  \caption{Results using \msolve}
  \label{tbl:msolve}
  \begin{center}
    \begin{tabular}{l|ll|ll}
      $\omega$ & \multicolumn{2}{c}{$\rho$} & \multicolumn{2}{c}{$\mu$} \\
      $\delta$  & \# real & time & \# real &  time\\
      \midrule
      3 & 0 & 0.7 & \multicolumn{2}{c}{$\dim > 0$}\\
      5 & 4 & 32.2 & \multicolumn{2}{c}{$\dim > 0$}\\
      7 & 4 & 5593.5 & \multicolumn{2}{c}{$\dim > 0$}\\
      9 & \multicolumn{2}{c}{?} & \multicolumn{2}{c}{$\dim > 0$}\\
    \end{tabular}
  \end{center}
\end{table*}

\begin{figure}[th] \centering
  \includegraphics[height=0.45\textwidth]{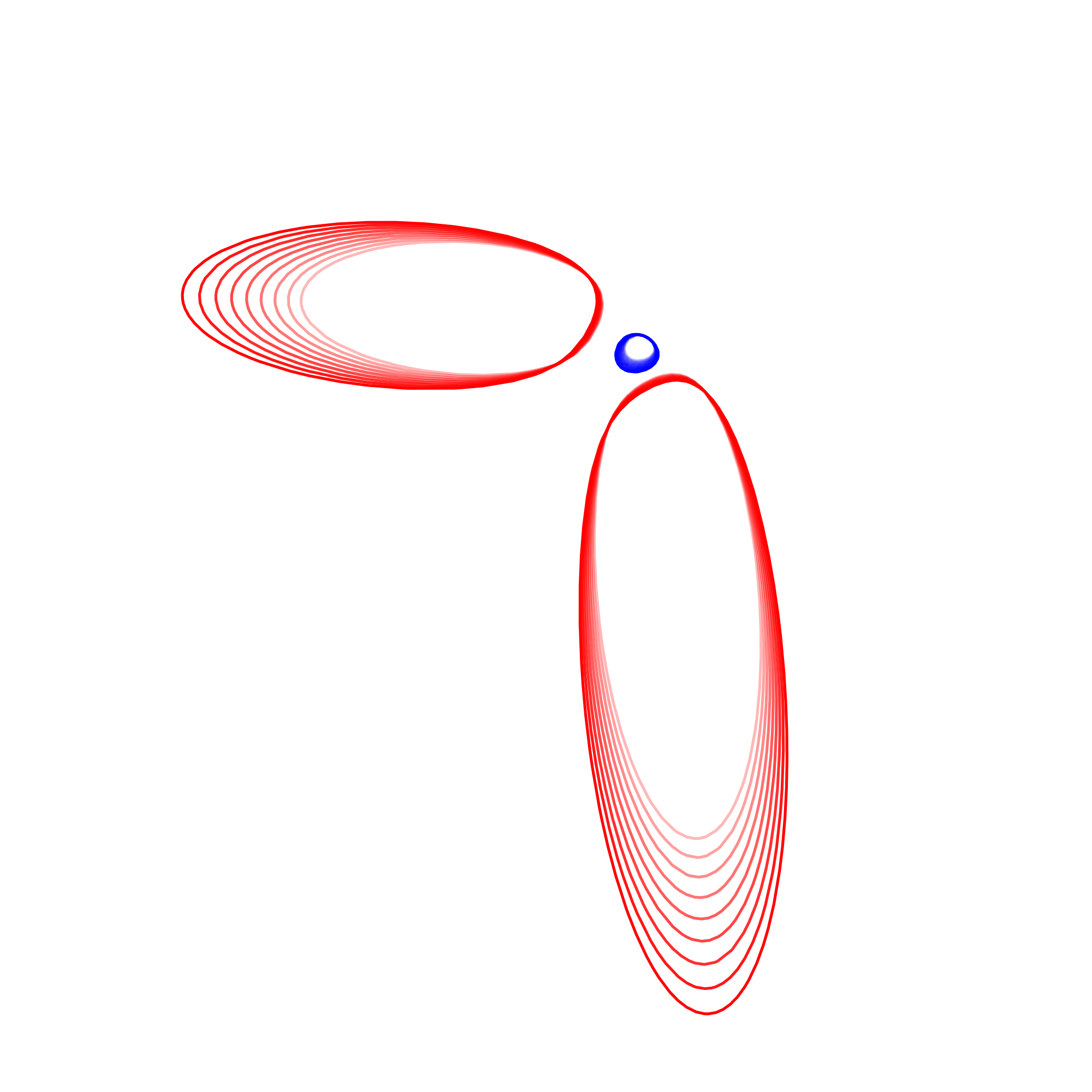}
  \includegraphics[height=0.45\textwidth]{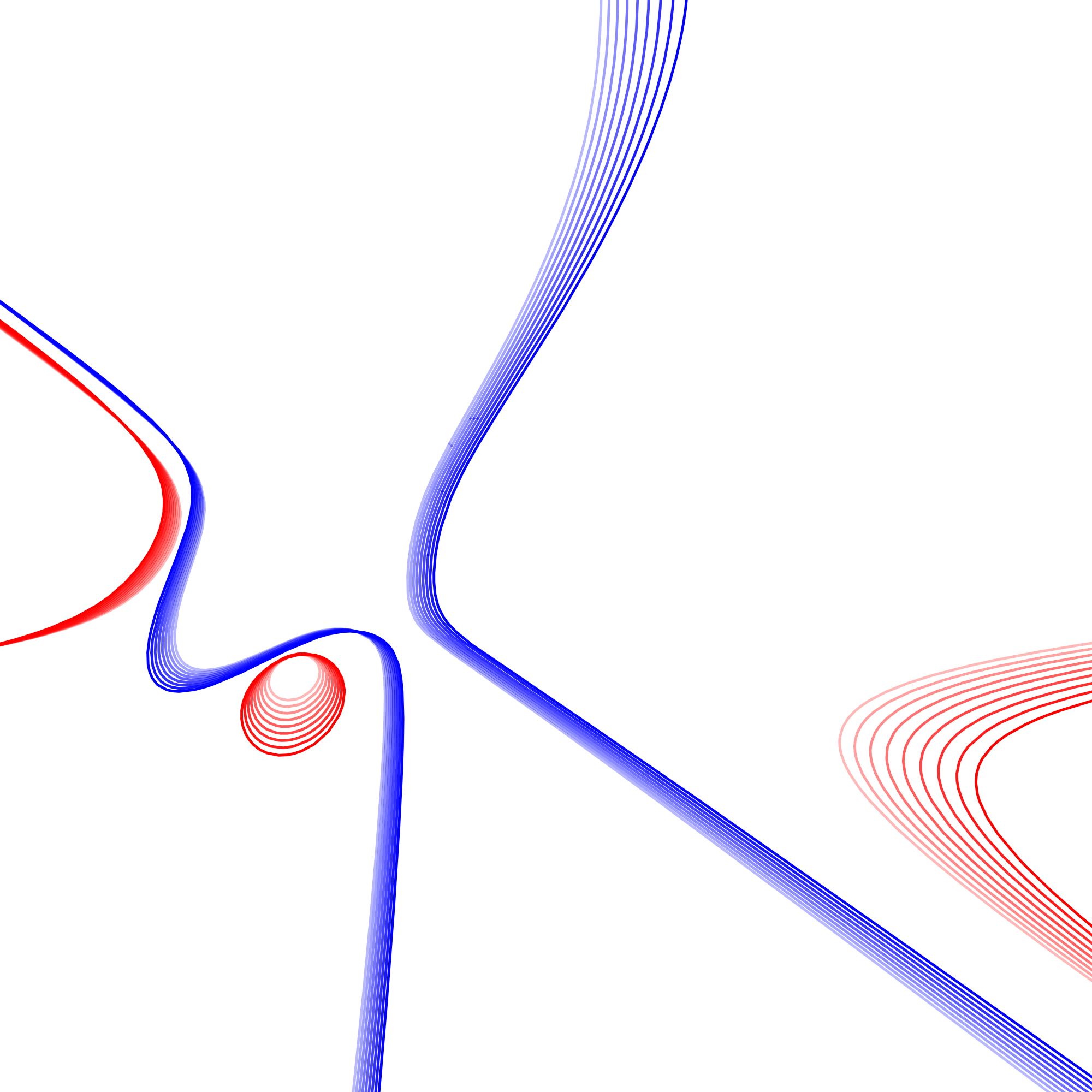}
	
  \caption{Two Wronski pairs of nonintersecting honeycomb quartic curves.}
  \label{fig:rhoquartics}
\end{figure}	

Finally, we attempted to compute the eliminant of the ideal corresponding to \eqref{eq:tosolve2} using \msolve for
$\omega = \rho$. This means more precisely
that we computed the unique monic generator of $I_{\delta}\cap \mathbb{Q}[t]$. This can be done with a Gröbner basis computation by
computing a Gröbner basis for $I_{\delta}$ with a monomial order eliminating $x$ and $y$. Again, \msolve does this
probabilistically using multi-modular methods. We then computed each real root of this eliminant. The data is summarized
in Table \ref{tbl:msolve2}. We note that these computations seem to be much better behaved than the computations
recorded in Table~\ref{tbl:msolve}. One possible explanation for this is that we observed that the bitsize of the
coefficients of these eliminants was relatively small. This means that \msolve requires only relatively few primes to
construct the result in a multi-modular way. Despite these improvements over the computation before we still note that the computations get
rapidly more difficult with growing $\delta$.

\begin{table*}[hbt!]
  \caption{Further results using \msolve}
  \label{tbl:msolve2}
  \begin{center}
    \begin{tabular}{l|lll}
      $\omega$ & \multicolumn{3}{c}{$\rho$} \\
      $\delta$ & degree & \# real & time \\
      \midrule
      3 & 6 & 0 & 1.1 \\
      5 & 60 & 2 & 0.0 \\
      7 & 204 & 4 & 8.61 \\
      9 & 360 & 4 & 400.9 \\
      11 & 969 & 13 & 49671.7 \\
    \end{tabular}
  \end{center}
\end{table*}

The results in Table \ref{tbl:msolve2} imply (probabilistically) that for all odd $3\leq \delta \leq 11$ the eliminant in $t$ is not zero.
We follow the convention common in numerical algebraic geometry to add the \enquote{*} to our theorem.
Here it indicates that the computations verifying it do not constitute a full proof but only that the result is probably true, with high certainty.

\begin{theorem*}
  For $\delta \leq 11$ and $\omega = \rho$ there exists $t_0> 0$ such that the system \eqref{eq:tosolve2} has no real solutions for any $0< t \leq t_0$.
  Consequently, the associated pairs of Wronski curves have at least $\delta$ real points of interection.
\end{theorem*}

We close with examples for $\delta = 4$ for which Theorem \ref{thm:Soprunova+Sottile} does not apply because the orientability assumption fails.
In \cite[\S8.3]{Sottile:2011} Sottile gives lower bounds for very special nonorientable systems.
Our examples suggest that it may be difficult to find a general nontrivial lower bound for nonorientable Wronski pairs; cf.\ \cite[\S7.4(1)]{Sottile:2011}.

\begin{example}
  \label{exmp:quartics}
  In Figure \ref{fig:rhoquartics}, we depict two families of Wronski pairs for $\delta = 4$.
  Since the latter number is even, the smooth locus of $Y_\cA^+$ is not orientable, due to Proposition~\ref{prop:orientability-triangle}.
  In both cases, the lifting function $\rho$ gives no real solutions.
  The parameters for the first Wronski pairs (on the left) are $(c_0, c_1, c_2)=(0.99,2.98,1.95)$, $(c_0',c_1',c_2')=(14.46, 1.57, 2.21)$.
  The second set of parameters (on the right) are $(c_0,c_1,c_2)=(-10.46,-1.07,9.43)$, $(c_0',c_1',c_2')=(12.62, 9.97, -0.86)$. In both cases $t$ ranges from 0.96 to 1.
\end{example}

\subsection*{Technical details} The computations were performed on a single core of an Intel Xeon Gold 6244 CPU @
3.60GHz with 1.5 TB of memory. We used version 2.11.0 of \HC and version 0.7.3 of \msolve via \OSCAR v1.2.0 \cite{oscar-book}.

\printbibliography

\end{document}